\documentclass[11pt]{article}

\usepackage{blindtext}

\usepackage{amsmath}
\usepackage{tikz}
\usepackage{caption}
\usepackage{subcaption}
\usepackage{amsthm}
\usepackage{amssymb}
\usepackage{natbib}
\usepackage{authblk}

\newtheorem{theorem}{Theorem}
\newtheorem{lemma}{Lemma}

\newtheorem{corollary}{Corollary}

\newcommand{\stargraph}[2]{\begin{tikzpicture}
    \node[circle,fill=black] at (360:0mm) (center) {};
    \foreach \n in {1,...,#1}{
        \node[circle,fill=black] at ({\n*360/#1}:#2cm) (n\n) {};
        \draw (center)--(n\n);
    }
\end{tikzpicture}}

\usepackage{lastpage}

\begin{document}

\title{Spatial depth for data in metric spaces}
\author{Joni Virta}
\affil{Department of Mathematics and Statistics\\University of Turku, Finland}


\maketitle

\begin{abstract}
We propose a novel measure of statistical depth, the metric spatial depth, for data residing in an arbitrary metric space. The measure assigns high (low) values for points located near (far away from) the bulk of the data distribution, allowing quantifying their centrality/outlyingness. This depth measure is shown to have highly interpretable geometric properties, making it appealing in object data analysis where standard descriptive statistics are difficult to compute. The proposed measure reduces to the classical spatial depth in a Euclidean space. In addition to studying its theoretical properties, to provide intuition on the concept, we explicitly compute metric spatial depths in several different metric spaces. Finally, we showcase the practical usefulness of the metric spatial depth in outlier detection, non-convex depth region estimation and classification.  
\end{abstract}

\textbf{Keywords}: Distance function, Object data, Outlier detection, Statistical depth

\section{Introduction}\label{sec:intro}

The purpose of this work is to propose a measure of \textit{depth} for data residing in a general metric space $(\mathcal{X}, d)$. By depth we refer to a function that assigns to each point $\mu \in \mathcal{X}$ a measure of centrality $D(\mu; P_X)$ with respect to a given probability distribution $P_X$ on $\mathcal{X}$. The larger $D(\mu; P_X)$ is, the more centrally located $\mu$ is with respect to the probability mass of $P_X$. And, vice versa, points with small depths can be seen as outlying in the view of $P_X$. Our reason for working in a general metric space $\mathcal{X}$ is that many modern applications produce data that are inherently non-Euclidean (functions, compositions, trees, graphs, rotations, positive-definite matrices, etc.), known also as \textit{object data}. Hence, devising methodology that works in arbitrary metric spaces allows for capturing all of these data types at once, see, e.g., \cite{bhattacharya2003large, lyons2013distance, dubey2019frechet, dubey2022modeling, dai2022tukey, virta2022sliced} for a (highly non-exhaustive) list of works with a similar viewpoint. A particularly important aspect of object data analysis is interpretability (since the data spaces themselves are often quite unusual and exotic) and this plays a key role in the current work as well. 

The study of depths in Euclidean spaces, i.e., when $\mathcal{X} = \mathbb{R}^p$ and $d$ is the Euclidean distance, has a long history. We do not attempt to paraphrase these developments here and refer the interested reader to \cite{zuo2000general, serfling2006depth, mosler2022choosing} instead. As a concrete example of a depth function, consider the \textit{lens depth} \citep{liu2011lens} $D_L(\mu; P_X)$ of a point $\mu$ in a Euclidean space with respect to a distribution $P_X$, defined as
\begin{align}\label{eq:lens_depth}
    D_L(\mu; P_X) := P_X( \| X_1 - X_2 \| \geq \max \{ \| X_1 - \mu \|, \| X_2 - \mu \| \} ),
\end{align}
where $X_1, X_2$ are drawn independently from $P_X$. That is, $D_L(\mu; P_X)$ gives the probability that the edge between $X_1$ and $X_2$ is the longest in the triangle formed by the points $X_1, X_2, \mu$. This event can be understood as the point $\mu$ being ``between'' $X_1$ and $X_2$ and, hence, the greater its probability, the more central the point $\mu$ has to be. Other examples of classical depths include the half-space depth \citep{tukey1975mathematics}, Oja depth \citep{oja1983descriptive}, the simplicial depth \citep{liu1990notion} and the spatial depth \citep{chaudhuri1996geometric, vardi2000multivariate}, which is also known as the $L_1$-depth and is treated in more detail below.

Inspection of the different depth measures reveals that they are typically heavily motivated by geometric arguments. As such, it is reasonable to expect that their underlying ideas would apply also if the data resides in an arbitrary metric space instead of a Euclidean space. This viewpoint for the lens depth was taken in \cite{cholaquidis2020weighted, geenens2023statistical}, where it was shown that the definition \eqref{eq:lens_depth} leads to a meaningful measure of depth for data residing in a general metric space $(\mathcal{X}, d)$ when the Euclidean distances in \eqref{eq:lens_depth} are replaced with $d$. In \cite{dai2022tukey}, an equivalent treatment was given to half-space depth. As a natural continuation to these works, the contribution of the current paper is to give a similar treatment to the classical spatial depth,
\begin{align}\label{eq:spatial_depth}
    D_S(\mu; P_X) := 1 - \left\| \mathrm{E} \left\{ \mathrm{sgn}(X - \mu) \right\} \right\|^2,
\end{align}
where $X \sim P_X$, $\mathrm{sgn}(x) := \mathrm{I}(x \neq 0) x/\| x \|$ and $\| \cdot \|$ denotes the Euclidean norm. Note that, strictly speaking, the definition \eqref{eq:spatial_depth} does not equal the classical definition of $L_1$-depth/spatial depth \citep{vardi2000multivariate} but, rather, the two are a monotone transformation apart. Our reason for using a definition that slightly differs from the standard is that this choice will greatly simplify the resulting expression for the metric space extension of $D_S$.

To give an idea of the contents of this paper, we next summarize the definition and the most interesting properties of our proposed metric space extension of \eqref{eq:spatial_depth}. Let $(\mathcal{X}, d)$ be a complete and separable metric space and let $P_X$ be a probability distribution on $\mathcal{X}$. Consider, for the sake of this demonstration, a simplified scenario where the distribution $P_X$ has no atoms, i.e., $P_X(X = \mu) = 0$ for all $\mu \in \mathcal{X}$. In such a case, our proposed \textit{metric spatial depth} of the point $\mu \in \mathcal{X}$ with respect to $P_X$ takes the form
\begin{align}\label{eq:main_concept_intro}
    D(\mu; P_X) = 1 - \frac{1}{2} \mathrm{E} \left\{ \frac{ d^2(X_1, \mu) + d^2(X_2, \mu) - d^2(X_1, X_2) }{d(X_1, \mu) d(X_2, \mu)} \right\},
\end{align}
where $X_1, X_2$ are drawn independently from $P_X$. Simple computation reveals that \eqref{eq:main_concept_intro} reduces to \eqref{eq:spatial_depth} when $d$ is taken to be the Euclidean distance.

While, from the definition \eqref{eq:main_concept_intro} alone, it is not at all clear what $D(\mu; P_X)$ is measuring in the case of a general metric, it actually turns out that $D(\mu; P_X)$ captures the geometric structure of $\mathcal{X}$ in a highly natural way. To describe this behavior, we first recall the concept of lines in metric spaces. Three points, $x_1, x_2, x_3 \in \mathcal{X}$, are said to be on a line if some permutation of them yields equality in the triangle inequality for the metric $d$. Given the points, $x_1, x_2, x_3 \in \mathcal{X}$, we denote by $L[x_1, x_2, x_3]$ the event that $d(x_1, x_3) = d(x_1, x_2) + d(x_2, x_3)$, i.e., that the three points are on a line such that $x_2$ is in the middle of the other two. Using this terminology, our Theorem \ref{theo:range} in Section \ref{sec:properties} says that $D(\mu; P_X)$ takes values in $[0, 2]$ and that:
\begin{align}
    D(\mu; P_X) = 0 \quad & \mbox{if and only if} \quad P_X(L[X_1, X_2, \mu] \cup L[X_2, X_1, \mu]) = 1, \label{eq:case_1_intro} \\
    D(\mu; P_X) = 2 \quad & \mbox{if and only if} \quad P_X(L[X_1, \mu, X_2]) = 1, \label{eq:case_2_intro}
\end{align}
where $X_1, X_2$ are drawn independently from $P_X$. The events $L[X_1, X_2, \mu]$ and  $L[X_1, \mu, X_2]$ have been visualized in Figure \ref{fig:intro}. These two extreme cases justify calling $D(\mu ; P_X)$ a measure of depth: For a point $\mu$ to have maximal depth it must by \eqref{eq:case_2_intro} almost surely reside between any two points drawn from $P_X$ (on a line going through them). It goes without saying that, for this to happen, $\mu$ has to be a very central/deep point of $P_X$. Moreover, intuition says that such a point $\mu$ might not actually exist in every given metric space. Indeed, we later show that, when $(\mathcal{X}, d)$ is a Euclidean space, no point may have metric spatial depth equal to~2 and that, in fact, the maximal depth of a point in a Euclidean space is~1. 


\begin{figure}[t!]
    \centering
    \includegraphics[width=1\textwidth]{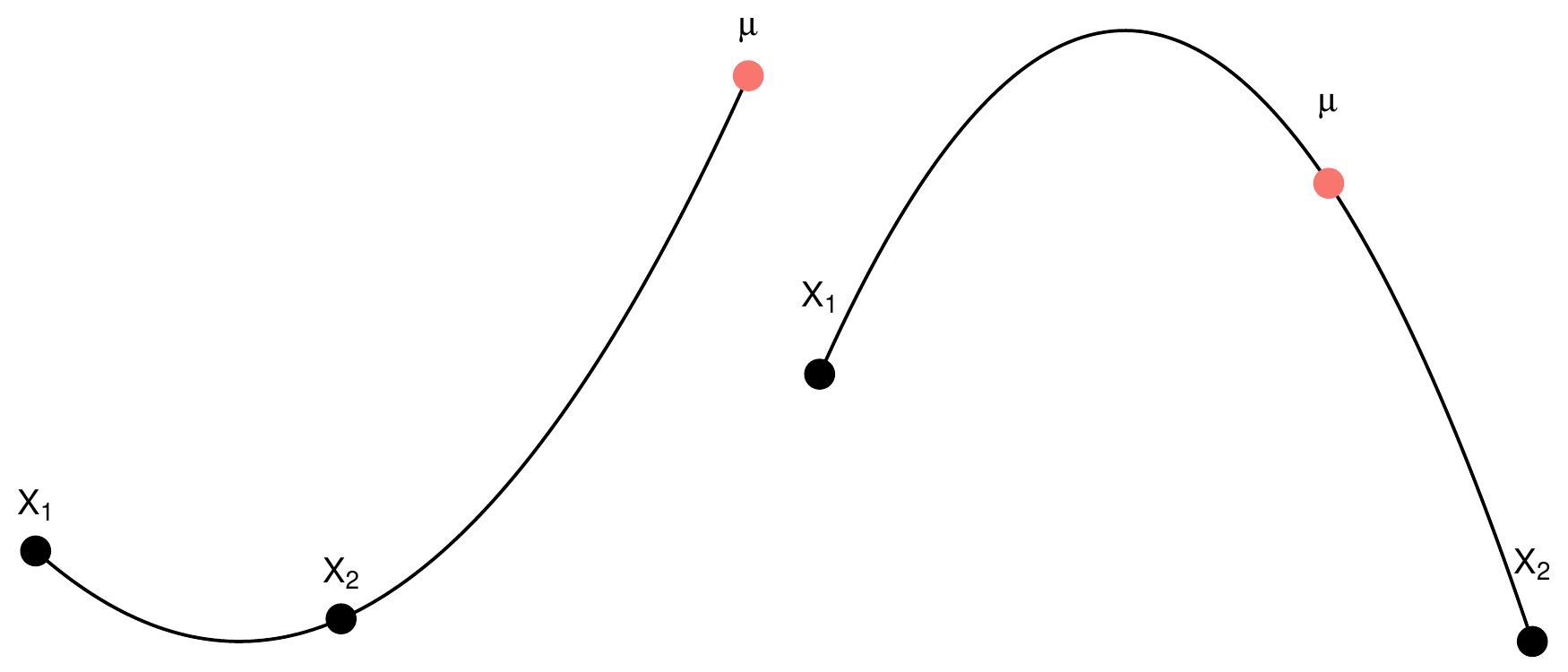}
    \caption{The left panel demonstrates the event $L[X_1, X_2, \mu]$ and the right panel the event $L[X_1, \mu, X_2]$. Curved lines have been used to emphasize that $(\mathcal{X}, d)$ does not have to be a Euclidean space.}
    \label{fig:intro}
\end{figure}

In a similar vein, for achieving the minimal depth, $D(\mu ; P_X) = 0$, by \eqref{eq:case_1_intro} the point $\mu$ may never be on a line between two points drawn from $P_X$, in the almost sure sense. Intuitively, any point satisfying this must lie sufficiently outside of the main probability mass of $P_X$, making it an outlier. The condition is perhaps most easily understood in an extreme case where $\mu$ is exceedingly far away from the bulk of the distribution $P_X$. Then $d(X_1, X_2)$ is negligible compared to $d(X_1, \mu)$ and $d(X_2, \mu)$ and, by ``zooming out'', we see that the three points are always approximately on a line, roughly satisfying the condition in \eqref{eq:case_1_intro} and implying that $\mu$ should have depth close to zero. This is indeed what happens, as is shown asymptotically later in Theorem \ref{theo:divergent_sequence} in Section \ref{sec:properties}. Finally, as pointed out earlier, properties such as \eqref{eq:case_1_intro} and \eqref{eq:case_2_intro} are highly desirable in object data analysis, where standard statistical intuition about Euclidean data no longer suffices.

While our work is targeted towards data sets lying in non-Euclidean spaces (as, in an Euclidean space, the proposed concept reduces to the standard spatial depth), it can still provide novel insights for Euclidean-valued data as well. Namely, as we demonstrate in Section \ref{sec:data_examples}, through the use of the kernel trick or graph distances, our proposed metric spatial depth allows combining a non-linear transformation with depth estimation for data in $\mathbb{R}^p$ in a natural way. This lets us in particular, achieve depth regions that are non-convex and adapt to non-standard data shapes, see Figure \ref{fig:nonlinear} in Section~\ref{sec:data_examples}. Similar ideas were used in the context of outlier detection in \cite{schreurs2021outlier}. In the same vein, our final example in Section \ref{sec:examples} shows how the metric spatial depth can be adapted to produce ``$L_p$-versions'' of the classical spatial depth (which is obtained when $p = 2$).


The main contributions of the current work are:
\begin{itemize}
    \item We extend the spatial depth to an arbitrary metric space and extensively study its robustness, interpretation, continuity and invariance properties (Section \ref{sec:properties}).
    \item We explicitly compute the metric spatial depth in four different scenarios, shedding further light on the intuitive meaning of the concept (Section \ref{sec:examples}).
    \item We discuss the properties, including root-$n$-consistency, of the sample metric spatial depth (Section \ref{sec:sample}).
    \item We apply the metric spatial depth to three different practical scenarios: outlier detection, non-convex depth region estimation and classification, also comparing it to both currently available metric depth concepts, the metric lens depth and metric half-space depth (Section~\ref{sec:data_examples}).
\end{itemize}

\section{Metric spatial depth}\label{sec:properties}


Let $(\mathcal{X}, d)$ be a complete and separable metric space and let $P_X$ be a probability distribution on $\mathcal{X}$. Throughout this work, we make the implicit assumption that every space is equipped with its Borel $\sigma$-algebra, guaranteeing, in particular, that all continuous functions are measurable. To define our proposed concept of depth, we first introduce the auxiliary function $h: \mathcal{X}^3 \rightarrow \mathbb{R}$,
\begin{align}\label{eq:h_function}
    h(x_1, x_2, x_3) := \mathbb{I}( x_3 \notin \{x_1, x_2\} ) \frac{ d^2(x_1, x_3) + d^2(x_2, x_3) - d^2(x_1, x_2) }{d(x_1, x_3) d(x_2, x_3) },
\end{align}
where $\mathbb{I}(\cdot)$ denotes the indicator function. Using $h$, we define the metric spatial depth $D(\mu ; P_X)$ of the point $\mu \in \mathcal{X}$ with respect to $P_X$ to be
\begin{align}\label{eq:main_concept}
    D(\mu; P_X) := 1 - \frac{1}{2} \mathrm{E} \{ h(X_1, X_2, \mu) \},
\end{align}
where $X_1, X_2 \sim P_X$ are independent. From the definition \eqref{eq:main_concept} alone it is not at all clear what kind of mapping $D$ is, and this section is thus devoted to studying the properties of $D(\mu; P_X)$. We divide the treatment to four categories: robustness, interpretation, continuity and invariance.

\subsection{Robustness of $D(\mu; P_X)$}

As with Euclidean statistical methods, also distance-based methodology commonly makes moment assumptions. In the latter case these are typically expressed as requirements of the form $\mathrm{E} \{ d^k(X_1, X_2) \} < \infty$ where $X_1, X_2 \sim P_X$ are independent and $k \in \mathbb{N}$. For example, this condition with $k = 2$ is in some sense equivalent to the requirement of finite variance (to which it reduces when $d$ is the Euclidean distance in $\mathbb{R}$). The next result shows that the metric spatial depth $D(\mu ; P_X)$ makes no such assumptions, making it a \textit{robust} measure of depth, applicable to any distribution $P_X$. We use $\mathcal{P}$ to denote the set of all probability distributions taking values in $\mathcal{X}$.
\begin{theorem}\label{theo:existence}
$D(\mu ; P_X)$ is finite for all $\mu \in \mathcal{X}$ and $P_X \in \mathcal{P}$.
\end{theorem}

Another classical requirement for calling a statistic robust is that its \textit{influence function} should be bounded \citep{hampel1986robust}. The influence function $IF: \mathcal{X} \to \mathbb{R}$ of $D(\mu; P_X)$ at $z \in \mathcal{X}$ is defined as,
\begin{align*}
    IF(z; D, \mu, P_X) := \lim_{\varepsilon \downarrow 0} \frac{D(\mu; (1 - \varepsilon) P_X + \varepsilon \delta_z) - D(\mu; P_X)}{\varepsilon},
\end{align*}
where $(1 - \varepsilon) P_X + \varepsilon \delta_z$ is a mixture of $P_X$ and $\delta_z$, a Dirac point mass at $z \in \mathcal{X}$. Thus, $IF$ measures the infinitesimal change in the metric spatial depth when the distribution $P_X$ is contaminated by a small probability mass at point $z$. For $D$ to be considered robust, the influence function must be bounded in $z$, i.e., there must be an upper limit for the effect that any contamination, no matter how outlying, is able to cause to the depth of the point $\mu$. Our next result reveals this to be the case.

\begin{theorem}\label{theo:influence_function}
    For fixed $z \in \mathcal{X}$, we have,
    \begin{align*}
        IF(z; D, \mu, P_X) = 2 - 2 D(\mu; P_X) - \mathrm{E} \{ h(X, z, \mu) \}.
    \end{align*}
    Moreover,
    \begin{align*}
        \mathrm{sup}_{z \in \mathcal{X}} | IF(z; D, \mu, P_X) | \leq 4.
    \end{align*}
\end{theorem}


\subsection{Interpretation of $D(\mu; P_X)$}

Recall next from Section \ref{sec:intro} that $L[x_1, x_2, x_3]$ denotes the event that $d(x_1, x_3) = d(x_1, x_2) + d(x_2, x_3)$, i.e., that the three points $x_1, x_2, x_3 \in \mathcal{X}$ fall in a line (in the sense of the metric $d$) such that $x_2$ is in the middle of the other two. Note also that this notation is not ``unique'' in the sense that, for example, $L[x_1, x_2, x_3] = L[x_3, x_2, x_1]$. We now have the following characterization about the range of possible values for $D(\mu; P_X)$.


\begin{theorem}\label{theo:range}
$D(\mu ; P_X)$ takes values in the interval $[0, 2]$. Additionally,
\begin{enumerate}
    \item[(i)] $D(\mu ; P_X) = 0$ if and only if
    \begin{align*}
        P_X(\{\mu\}) = 0 \quad \mbox{and} \quad P_X ( L[X_1, X_2, \mu] \cup L[X_2, X_1, \mu] ) = 1.
    \end{align*}
    \item[(ii)] $D(\mu ; P_X) = 2$ if and only if
    \begin{align*}
        P_X(\{\mu\}) = 0 \quad \mbox{and} \quad P_X ( L[X_1, \mu, X_2] ) = 1.
    \end{align*}
\end{enumerate}
\end{theorem}

Part (i) of Theorem \ref{theo:range} says that for a point $\mu$ to have depth equal to zero, a necessary condition is that the distribution $P_X$ must be concentrated on a line containing $\mu$. The standard case where this occurs is when $\mathcal{X} = \mathbb{R}$ (meaning that any three points always fall onto a line) and $P_X$ is supported on an interval. Any point $\mu$ outside of this interval then has depth equal to zero. Part (ii), on the other hand, says that maximal depth $D(\mu ; P_X) = 2$ is achieved by any non-atom point $\mu$ such that, given two independent realizations $X_1, X_2$, the point $\mu$ always falls between these two points on a line going through all three. This condition is very strict and can be satisfied only in some specific metric spaces veering towards pathological, see Section \ref{subsec:example_rail} for an example. Interestingly, in an Euclidean space, the maximal achievable depth is $D(\mu ; P_X) = 1$, see Section~\ref{subsec:example_hilbert}.

Theorem \ref{theo:range} provides an interesting connection between the metric spatial depth and the metric lens depth \cite{cholaquidis2020weighted, geenens2023statistical}, defined as,
\begin{align}\label{eq:metric_lens_depth}
   D_L(\mu ; P_X) =  P_X[ d(X_1, X_2) \geq \max \{ d( X_1, \mu), d( X_2, \mu) \} ],
\end{align}
where $X_1, X_2 \sim P_X$ are independent. It is obvious that $D_L$ takes values in the interval $[0, 1]$ and our next result connects these end points to the extremal behaviour of the metric spatial depth. Its proof is a direct consequence of Theorem \ref{theo:range} and we omit it.

\begin{corollary}\label{theo:lens_connection}
Let $\mu \in \mathcal{X}$ be fixed.
\begin{enumerate}
    \item[(i)] If $D(\mu ; P_X) = 0$ then $D_L(\mu ; P_X) = 0$.
    \item[(ii)] If $D(\mu ; P_X) = 2$ then $D_L(\mu ; P_X) = 1$.
\end{enumerate}
\end{corollary}

Corollary \ref{theo:lens_connection} essentially says that the range of behaviors that $D$ considers extremal is a subset of the behaviors that $D_L$ sees as extremal. For example, if the former marks the point $\mu$ as fully outlying, then so does the latter (but not necessarily vice versa). Note that this does not automatically mean that $D_L$ is a ``better'' measure of depth than $D$. On the contrary, we argue that the metric spatial depth $D$ can actually be seen as the ``safer'' choice of the two as Theorem \ref{theo:range} exactly quantifies what its extremal values measure. Note also that Corollary \ref{theo:lens_connection} does not say anything about the non-extremal behavior of the two depths, which can thus differ (however, when we compare the two later using simulations in Section \ref{sec:data_examples}, they do perform quite similarly).

We next turn our attention to an asymptotic version of Theorem \ref{theo:range}(i). Let $ \mu_n$ be a sequence of points in $\mathcal{X}$. We say that $\mu_n$ \textit{diverges} if $d(\mu_n, \mu) \rightarrow \infty$ as $n \rightarrow \infty$ for any fixed point $\mu \in \mathcal{X}$ (triangle inequality shows that if this condition holds for one point in $\mathcal{X}$, then it holds for all points). That is, divergent sequences are such that they eventually get arbitrary far from any fixed point and, naturally, the existence of such a sequence implies that the space $(\mathcal{X}, d)$ is unbounded.

\begin{theorem}\label{theo:divergent_sequence}
Let $\mu_n$ be a divergent sequence in $\mathcal{X}$. Then $D(\mu_n; P_X) \rightarrow 0$ as $n \rightarrow \infty$.
\end{theorem}

The result of Theorem \ref{theo:divergent_sequence} is very unsurprising when one considers it as a ``limiting case'' of part (i) of Theorem \ref{theo:range}: when the diverging sequence $\mu_n$ gets farther and farther away from the bulk of the distribution $P_X$, the observations $X_1, X_2$ and $\mu_n$ become closer and closer to being on a line, and the depth $D(\mu_n; P_X)$ approaches the minimal value given in Theorem \ref{theo:range}.

\subsection{Continuity properties of $D(\mu; P_X)$}

As our next two results, we show that the depth $D$ is continuous both with respect to the point $\mu$ and the distribution $P_X$. As the definition of $D$ involves division by $d(X_1, \mu) d(X_2, \mu)$, these results necessarily require that $P_X$ has no mass at the point $\mu \in \mathcal{X}$ in question. In the following, the notation $\rightsquigarrow$ denotes the weak convergence of probability distributions.

\begin{theorem}\label{theo:continuity_1}
    Fix $\mu \in \mathcal{X}$ and assume that $P_X(\{ \mu \}) = 0$. Let $P_{n}$ be a sequence of probability distributions in $\mathcal{X}$ such that $P_{n} \rightsquigarrow P_X$ as $n \rightarrow \infty$. Then, as $n \rightarrow \infty$, we have
    \begin{align*}
        D(\mu; P_{n}) \rightarrow D(\mu; P_X).
    \end{align*}
\end{theorem}

We demonstrate the use of Theorem \ref{theo:continuity_1} in Section \ref{subsec:example_circle} to compute the depths of points in the unit circle under the arc length distance through a limiting argument. 


\begin{theorem}\label{theo:continuity_2}
    Fix $\mu \in \mathcal{X}$ and assume that $P_X(\{ \mu \}) = 0$. Let $\mu_n$ be a sequence in $\mathcal{X}$ such that $\mu_n \rightarrow \mu$ in the metric $d$ as $n \rightarrow \infty$. Then, as $n \rightarrow \infty$, we have
    \begin{align*}
        D(\mu_n; P_X) \rightarrow D(\mu; P_X).
    \end{align*}
\end{theorem}

Theorem \ref{theo:continuity_2} implies, in particular, that the metric spatial depth is continuous in the whole of $\mathcal{X}$ as long as $P_X$ has no atoms. We omit the proof of Theorem \ref{theo:continuity_2} as the result follows from a direct application of Lebesgue's dominated convergence theorem. Later, when discussing the sample version of $D$ in Section \ref{sec:sample}, we present still one more continuity result for the metric spatial depth (root-$n$ consistency).




\subsection{Invariance properties of $D(\mu; P_X)$}

We conclude the section by discussing invariance properties of the metric spatial depth. Assume that the metric $d$ is invariant to a group $\mathcal{G}$ of transformations $g: \mathcal{X} \to \mathcal{X}$. That is, $d(gx_1, gx_2) = d(x_1, x_2)$ for all $g \in \mathcal{G}$, $x_1, x_2 \in \mathcal{X}$. Then, we clearly have $D(g \mu; P_{gX}) = D(\mu; P_{X})$ for any $g \in \mathcal{G}$ where $P_{gX}$ denotes the distribution of $gX$. We give a few concrete examples: (i) If $(\mathcal{X}, d)$ is either the $p$-dimensional Euclidean space or the $p$-dimensional unit sphere equipped with the arc length metric, then the depth $D$ is invariant (in the previous sense) to orthogonal transformations. (ii) If $\mathcal{X} = \{ 0, 1 \}^p$ and $d$ is the Hamming distance, then $D$ is invariant to any bit-flip operations. (iii) If $\mathcal{X}$ is the space $\mathcal{S}^p$ of $p \times p$ positive-definite symmetric matrices and $d$ is the corresponding Riemannian distance \citep{bhatia2009positive}, then $D$ is invariant to any transformations of the form $S \mapsto A S A'$ where $S \in \mathcal{S}^p$ and $A \in \mathbb{R}^{p \times p}$ is invertible.

Assume now that $P_X$ is such that $P_X = P_{gX}$ for all $g \in \mathcal{G}$. Then, by the previous paragraph, we have $D(g_1 \mu, P_{X}) = D(g_2 \mu, P_{X})$ for all $g_1, g_2 \in \mathcal{G}$. I.e., the metric spatial depth is constant on the orbit $\{ g \mu \mid g \in \mathcal{G} \}$ of any point $\mu \in \mathcal{X}$. As an example of this, we continue our example scenario (i) above: any $p$-variate random vector $X$ in a Euclidean space satisfying $X \sim OX$ for all $p \times p$ orthogonal matrices $O$ is said to be spherical \citep{fang1990symmetric} As the orbits of the group of orthogonal transformations are origin-centered hyperspheres, the contours of the metric spatial depth $D$ for a spherical distribution are thus sphere-shaped around the origin.

We note that the previous properties are somewhat trivial, being inherited directly from the metric $d$ and not from our definition of the metric spatial depth $D$. However, it is also obvious that without imposing a stronger structure on the metric space $(\mathcal{X}, d)$ they are the best that we can hope to achieve.


\section{Four example scenarios}\label{sec:examples}

In this section, we study the properties of the metric spatial depth $D(\mu; P_X)$ in four specific metric spaces. We stress that all of these example scenarios might not necessarily be terribly relevant in statistical practice, but rather their purpose here is to (i) give further insight about the behavior and interpretation of the metric spatial depth, and (ii) demonstrate that our proposed concept allows for closed-form solutions in several interesting cases.

\subsection{Hilbert space}\label{subsec:example_hilbert}

Let $(\mathcal{X}, \langle \cdot, \cdot \rangle)$ be a Hilbert space and let $\| \cdot \|$ and $d$ be the norm and the metric induced by the inner product, respectively. The sign function $\mathrm{sgn}: \mathcal{X} \to \mathcal{X}$ is defined as $\mathrm{sgn}(x) = \mathrm{I}(x \neq 0) x/\| x \|$. Let $X$ be a random element in $\mathcal{X}$ satisfying $\mathrm{E} \| X \| < \infty$. Then we define the expected value of $X$ in the usual way through the Riesz representation theorem \citep{conway1990course}, as the unique element $\mathrm{E}(X) \in \mathcal{X}$ satisfying $\mathrm{E} \langle X, \mu \rangle = \langle \mathrm{E}(X), \mu \rangle$ for all $\mu \in \mathcal{X}$.

\begin{lemma}\label{lem:inner_product}
For a Hilbert space $(\mathcal{X}, \langle \cdot, \cdot \rangle)$, the metric spatial depth takes the form,
\begin{align*}
    D(\mu; P_X) = 1 - \left\| \mathrm{E} \left\{ \mathrm{sgn}(X - \mu) \right\} \right\|^2.
\end{align*}
\end{lemma}

Lemma \ref{lem:inner_product} carries with it several insights: (i) The expression $D(\mu; P_X)$ is based on the same key quantity as the classical $L_1$-depth \citep{vardi2000multivariate}. I.e., the depth of a point is determined by the expected length of an unit vector drawn from it toward a point generated randomly from $P_X$, see Figure~\ref{fig:L1_depth} for an illustration. The $L_1$-depth is also known as the spatial depth, see, e.g., \cite{serfling2006depth}, justifying naming our proposed concept the ``metric spatial depth''. (ii) For a distribution $P_X$ on the real line (1-dimensional Euclidean space), the metric spatial depth is seen to further reduce to
\begin{align*}
    D(\mu; P_X) = 1 - \{ P(X > \mu) - P(X < \mu) \}^2,    
\end{align*}
that is, the depth of a point $\mu$ is in this case fully determined by its quantile level. As discussed in \cite{geenens2023statistical}, this is desirable since $\mathbb{R}$ has a natural concept of order (quantiles) and, hence, any reasonable measure of depth therein should be based on this order structure. (iii) In a Hilbert space, such as $\mathbb{R}^p$ equipped with the Euclidean inner product, the maximal achievable depth is~1. This value is obtained if and only if $\mu$ is the center of the distribution $P_X$ in the sense that the unit length vector drawn from it towards $X$ has on average length $0$.

\begin{figure}[t]
    \centering
    \includegraphics[width=0.9\textwidth]{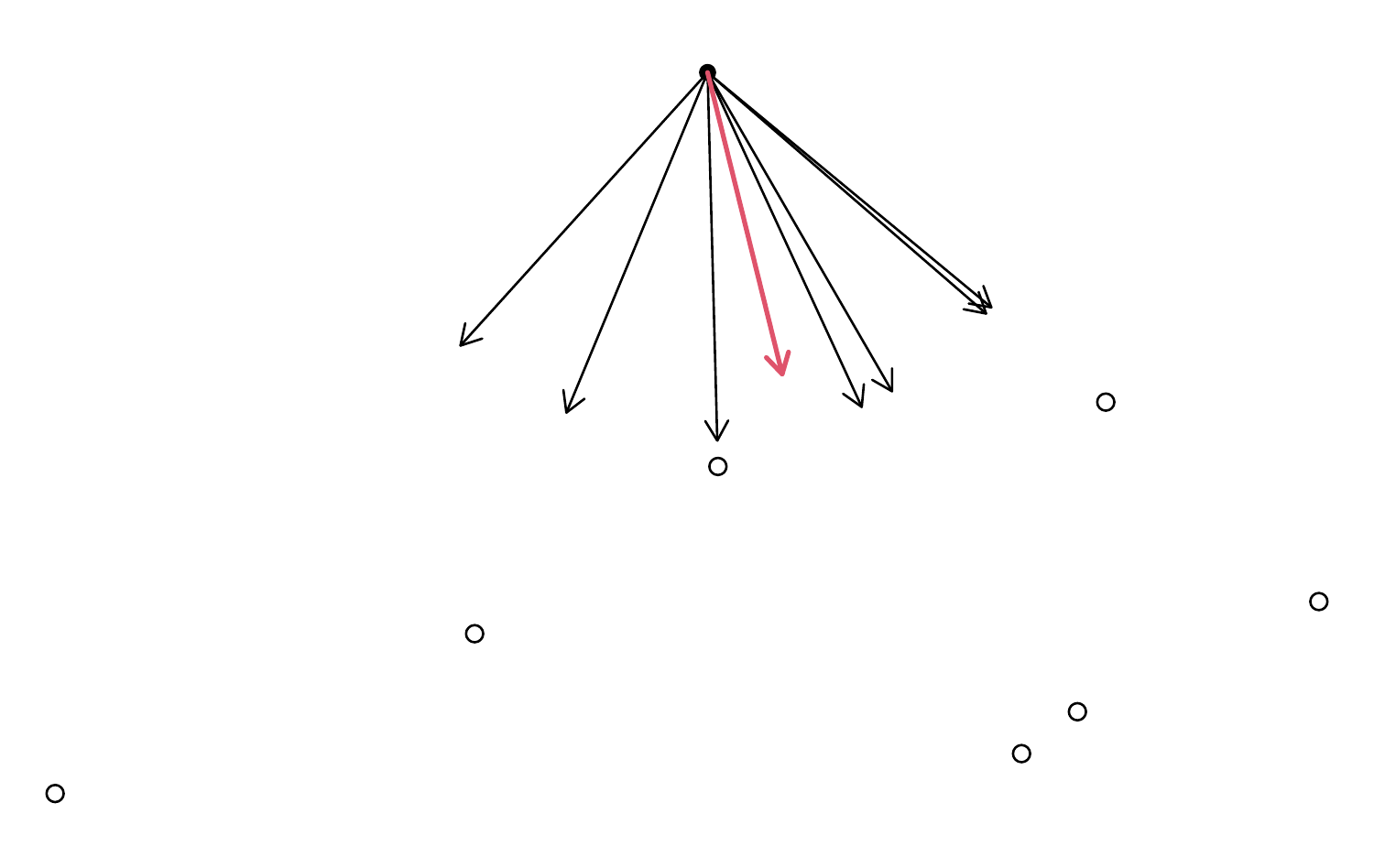}
    \caption{Assume that $P_X$ puts equal mass to all seven circles shown in the plot. The black arrows are unit length vectors drawn from $\mu$ (the solid black point) towards these point masses. The red shorter arrow depicts the average of these vectors and its length ($\approx 0.843$) determines the depth of $\mu$, larger values indicating smaller depths. As $\mu$ lies outside of the sample, a rather small depth, $D(\mu; P_X) \approx 0.289$, is now obtained.}
    \label{fig:L1_depth}
\end{figure}

\subsection{British rail metric}\label{subsec:example_rail}

As our next example, we consider a more pathological case where the theoretical maximum depth value 2 can be reached. Let $\mathcal{X} = \mathbb{R}^2$ be the plane and take the metric $d$ to be the British rail metric,
\begin{align*}
    d(x_1, x_2) = \mathbb{I}(x_1 \neq x_2) (\| x_1 \| + \| x_2 \|),
\end{align*}
where $\| \cdot \|$ denotes the Euclidean norm. This metric space models the situation where all train trips in the Southern England have to go through London (located in the origin).
Let next $P_X$ be any distribution in $\mathbb{R}^2$ having no atoms. Then $d(X_1, X_2)$ equals almost surely $\| X_1 \| + \| X_2 \|$ and the depth of a point $\mu \in \mathcal{X}$ takes the form
\begin{align*}
    D(\mu; P_X) :=  \mathrm{E} \left\{ \frac{ 2 \| X_1 \| \| X_2 \|  }{(\| X_1 \| + \| \mu \|)(\| X_2 \| + \| \mu \|) } \right\},
\end{align*}
showing that $D(0; P_X) = 2$.

The British rail metric appears also as the solution to the following optimality problem: Assume that $P_X$ puts equal mass $1/n$ to each of $n$ distinct points $x_1, \ldots, x_n \in \mathcal{X}$. Then, under what metric is  $D(x_1; P_X)$ maximized? In other words, which combination of distances between  $n$ equiprobable distinct points makes $x_1$ as ``central'' as possible? The following lemma reveals that the answer is given by the British rail metric.

\begin{lemma}\label{lem:rail_optimality}
    Let $P_X$ put mass $1/n$ to each of $n$ distinct points $x_1, \ldots, x_n \in \mathcal{X}$. Then,
    \begin{align*}
        D(x_1; P_X) \leq 1 + \left(1 - \frac{1}{n} \right)\left(1 - \frac{3}{n} \right),
    \end{align*}
    with equality if and only if $d(x_i, x_j) = d(x_i, x_1) + d(x_1, x_j)$ for all distinct $i, j$ in $\{ 2, \ldots, n \}$.
\end{lemma}

\begin{figure}
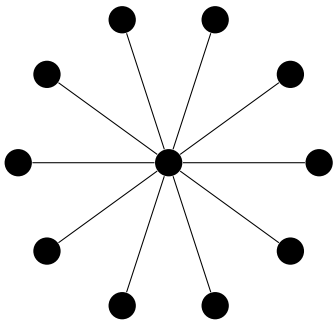

    \centering
    \stargraph{10}{2} 
\caption{The picture displays a star graph on $n = 11$ points under which $x_1$ (the central point in the graph) achieves the maximal depth $D(x_1; P_X) = 2 - 41/121 \approx 1.661$. The distances in the graph are proportional to the edges, each of which has the same length.}
    \label{fig:star_graph}
\end{figure}

The equality condition in Lemma \ref{lem:rail_optimality} corresponds to a situation where the shortest route between any two points always goes through $x_1$ (``London''). Interestingly, the actual distances between the points do not play a role in the solution, as long as they are positive. One particular, symmetric solution achieving the upper bound has been visualized in Figure~\ref{fig:star_graph}. As is expected, the upper bound in Lemma \ref{lem:rail_optimality} approaches the value $2$ as $n \rightarrow \infty$.


\subsection{Discrete metric in a finite space}\label{subsec:example_discrete}

To gain further intuition on the metric spatial depth, we consider one of the simplest possible spaces, a finite set $\mathcal{X} := \{ 1, \ldots, n \}$ equipped with the discrete metric $d(i, j) = \mathbb{I}(i \neq j)$. Any probability distribution $P_X$ on $\mathcal{X}$ is then characterized by the probabilities $p_1, \ldots , p_n$ of the $n$ points. Our next result gives the values of the metric spatial depth in this scenario.

\begin{lemma}\label{lem:finite_set}
    In a finite set $\mathcal{X} = \{ 1, \ldots, n \}$ equipped with the discrete metric $d(i, j) = \mathbb{I}(i \neq j)$, we have
    \begin{align*}
        D(1;P_X) = \frac{1}{2} \left( 1 - \sum_{i = 1}^n p_i^2 \right) + p_1 \in [0, 1].
    \end{align*}
    Furthermore,
    \begin{itemize}
        \item[(i)] $D(1; P_X) = 0$ if and only if exactly one of $p_2, \ldots , p_n$ equals $1$.
        \item[(ii)] $D(1; P_X) = 1$ if and only if $p_1 = 1$.
    \end{itemize}
\end{lemma}

Lemma \ref{lem:finite_set} shows that the extremal depth values are obtained only in cases where the probability mass is concentrated solely at a single point. This is expected as the only way to make a point outlying under a discrete metric where all distances are of equal length, is to rob it of its probability mass. Similarly, as ``location'' carries no meaning under a discrete metric, a point's centrality is measured solely in its (relative) probability mass.

\subsection{Uniform distribution on the unit circle}\label{subsec:example_circle}

Our final example involves the metric space $(\mathbb{S}^1, d)$, studied commonly in directional statistics \citep{mardia2000directional}, where $\mathbb{S}^1$ is the unit circle in $\mathbb{R}^2$ and $d$ is the arc length distance. The next result gives a closed form solution to the metric spatial depth under the assumption of a uniform distribution on $\mathbb{S}^1$. As integration of the arc length distance function $d$ is a non-trivial task, the proof instead relies on constructing a sequence of discrete uniform distributions on the circle and invoking Theorem~\ref{theo:continuity_1}.

\begin{theorem}\label{theo:cycle_graph}
Let $P_X$ be the uniform distribution on the unit circle $\mathbb{S}^1$. Then, for any point $x \in \mathbb{S}^1$, we have,
\begin{align*}
    D(x ; P_X) = \pi^2/6 - 1 \approx 0.6449.
\end{align*}
\end{theorem}



\section{Sample metric spatial depth}\label{sec:sample}

Let $X_1, \ldots, X_n$ be a random sample from the distribution $P_X$. Denote the empirical distribution of the sample by $P_n$. The sample metric spatial depth $D(\mu; P_n)$ of the point $\mu \in \mathcal{X}$ is given by
\begin{align*}
    D(\mu; P_n) = 1 - \frac{1}{2 n^2} \sum_{i, j = 1}^n h(X_i, X_j, \mu),
\end{align*}
where $h$ is defined as in \eqref{eq:h_function}. By isolating the terms with $i \neq j$ in the double sum, we see that $D(\mu; P_n)$ is asymptotically equivalent to a second order U-statistic, guaranteeing that $ D(\mu; P_n)$ is a root-$n$-consistent estimator of the population depth $D(\mu; P_X)$.

\begin{theorem}
    For a fixed $\mu \in \mathcal{X}$, we have,
    \begin{align*}
        D(\mu; P_n) = D(\mu; P_X) + \mathcal{O}\left( \frac{1}{\sqrt{n}} \right),
    \end{align*}
    as $n \rightarrow \infty$.
\end{theorem}

The U-statistic theory \citep{lee1990u} further guarantees that $\sqrt{n} \{ D(\mu; P_n) - D(\mu; P_X) \}$ admits a limiting normal distribution. Pursuing such finer asymptotic properties of $D(\mu; P_n)$ would allow, e.g., testing hypotheses of the form $H_0: D(\mu; P_X) = \alpha$, for some point $\mu$ and depth level $\alpha \in [0, 2]$. However, such tests are rather impractical in the general situation as the range of $D(\mu, P_X)$ depends on the metric space in question, see Section \ref{subsec:example_hilbert} for an example. Thus, to test, for instance, whether $\mu \in \mathcal{X}$ is the deepest point w.r.t. $P_X$ necessarily requires knowing the minimal value of the map $\mu \mapsto D(\mu; P_X)$. Consequently, in this work, to stay as general as possible, we have chosen not to study such questions. 




In practice, for a realized sample $X_1, \ldots, X_n$, computing $D(\mu; P_n)$ for a single value of $\mu$ requires two nested loops over the sample, leading to time complexity of order $\mathcal{O}(n^2)$. As there does not appear to exist a straightforward connection between the terms involved in computing, e.g., $D(X_1, P_n)$ and $D(X_2, P_n)$, finding the depths of the full sample is thus an $\mathcal{O}(n^3)$-operation. Such complexities are rather standard for methods based on distance matrices, see, e.g., \cite{cholaquidis2020weighted, dai2022tukey}.



\section{Examples}\label{sec:data_examples}

We next exemplify the use of metric spatial depth in three tasks: outlier detection, non-convex depth region estimation and classification. The used methods are implemented in \texttt{R} and their program codes are available on the author's webpage.

\subsection{Outlier detection}

In our outlier detection task, we study how well different methods are able to separate outliers from the bulk of a distribution. We generate $10$-dimensional points as $x_i \sim \mathcal{N}_{10}(\lambda 1_{10}, \mathrm{diag}(I_{10}))$ and transform them as $x_i \mapsto x_i/\| x_i \|$, meaning that the data lies on the unit sphere in $\mathbb{R}^{10}$. To capture the geometry of the space, we use the arc length distance as our metric. We create two groups, the bulk of size $(1 - \varepsilon)n$ and the outlier group of size $\varepsilon n$ where the value $-\lambda$ is used for the location of the outlying group in place of $\lambda$ (i.e., the bulk and the outlier group reside on the opposite sides of the sphere). We consider three different values of lambda, $\lambda = 1/4, 1/3, 1/2$, two different sample sizes, $n = 50, 100$ and a range of different proportions of outliers, $\varepsilon = 0.01, 0.02, \ldots, 0.30$. A similar setting was used in  \cite{heinonen2023robust}.


\begin{figure}[t]
    \centering
    \includegraphics[width=1\textwidth]{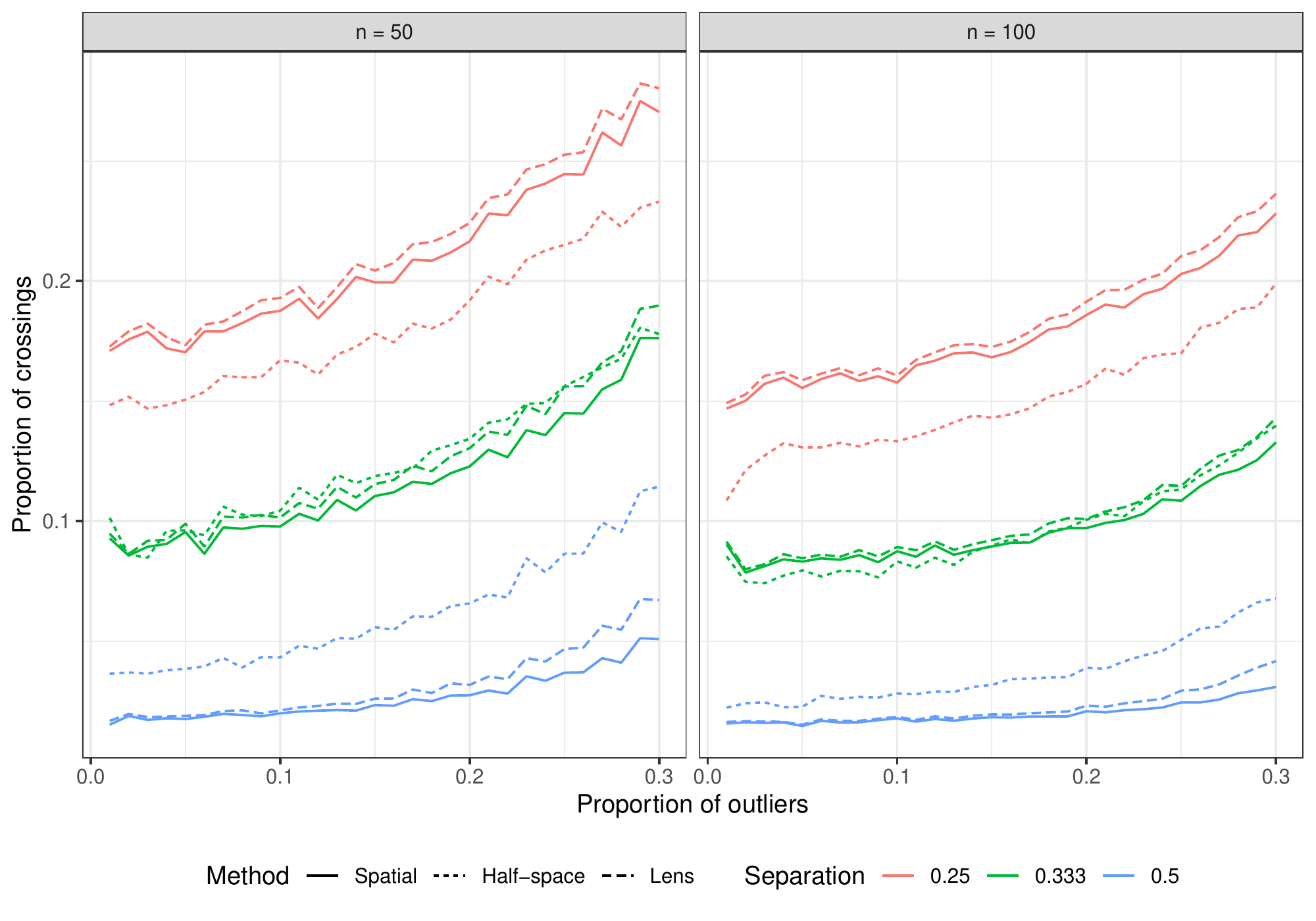}
    \caption{Proportions of bulk-outlier crossings $C$ (see the main text for the definition) for the metric spatial, half-space and lens depths under different settings over 1000 replications. A lower value indicates a better separation.}
    \label{fig:outlier_simu}
\end{figure}

We conduct 1000 replicates of each setting and in each case compute our metric spatial depth \eqref{eq:main_concept}, the metric lens depth \eqref{eq:metric_lens_depth} \citep{cholaquidis2020weighted, geenens2023statistical} and the metric half-space depth \citep{dai2022tukey}, defined as
\begin{align*}
    D_H(\mu; P_X) := \inf_{z_1, z_2 \in \mathcal{X}, d(z_1, \mu) \leq d(z_2, \mu)}P_X\{ d(X, z_1) \leq d(X, z_2) \}.
\end{align*}
We compute $D_H$ using Algorithm 1 in \cite{dai2022tukey} using all sample pairs as anchor points. All three competitors output the set of depth values for the full sample and, for each set, we compute the average number $C$ of ``bulk-outlier -crossings'', i.e., the proportion of pairs of a bulk observation and an outlying observation such that the latter gets a lower depth than the former. Thus, the smaller $C$ is, the better the depths identify/separate the outlying observations from the bulk. The average values of $C$ for each combination of settings and method are shown in Figure \ref{fig:outlier_simu}. 

From Figure \ref{fig:outlier_simu} we make the following conclusions: (i) The performances of the spatial and lens depth are in all scenarios very close to each other. Based on Theorem \ref{theo:lens_connection} this is as expected. Moreover, the spatial depth is systematically the better one out of the two. We believe that this is connected to the fact that, by Theorem \ref{theo:range}, low values of the metric spatial depth characterize a very natural form of outlyingness, something which is not guaranteed for the metric lens depth. (ii) As is natural, independent of the method, separation gets easier the larger the sample size is, the less we have outliers and the greater the separation between the groups is. (iii) The best two methods are the spatial depth and the half-space depth. The spatial depth has an edge when the bulk and the outliers are sufficiently separated, whereas the half-space depth dominates when the groups are close to each other. In both of the extreme cases $\lambda = 1/2, 1/4$, the difference between the two methods is rather large. Again we believe that this difference can be accounted by Theorem~\ref{theo:range}: large values of the the metric spatial depth indicate a very specific form of centrality, which is likely ill-suited to this particular scenario, meaning that the method encounters difficulties when the outliers are brought closer to the center of the bulk.







\subsection{Depth and non-convex distributions}

A common feature of several classical measures of depth in $\mathbb{R}^p$ is that they produce convex depth regions \citep{serfling2006depth}. Naturally, this is undesirable as soon as the data exhibits non-convex shapes. In this section we show how the proposed metric spatial depth can be adapted to obtain meaningful measures of depth also in such scenarios.

For illustration purposes, we use a simple bivariate dataset of size $n = 150$, generated as $x_i = 2 (\cos(\theta_i), \sin(\theta_i) ) + 0.1^{1/2} \varepsilon$, where $\theta_i$ are i.i.d. from the uniform distribution on $[0, 2 \pi]$ and $\varepsilon$ are i.i.d. from the bivariate standard normal distribution. This yields a sample resembling a circle around the origin on the plane, see the black points in the panels of Figure \ref{fig:nonlinear}. To exemplify the behavior of standard measures of depth under such a non-convex scenario, we divided the plane into a fine grid and computed the metric spatial depth of each grid point under the Euclidean metric. As described in Section \ref{subsec:example_hilbert}, this yields depths equivalent to the classical $L_1$-depth. The contour plot of the depths, shown in the upper left panel of Figure \ref{fig:nonlinear} (yellow colour indicates large values), reveals that the depths indeed fail to capture the natural shape of the sample, assigning the highest depth around the origin where no sample points lie. To obtain better performance, we next propose two alternative approaches.

\begin{figure}[t!]
    \centering
    \includegraphics[width = 1\textwidth]{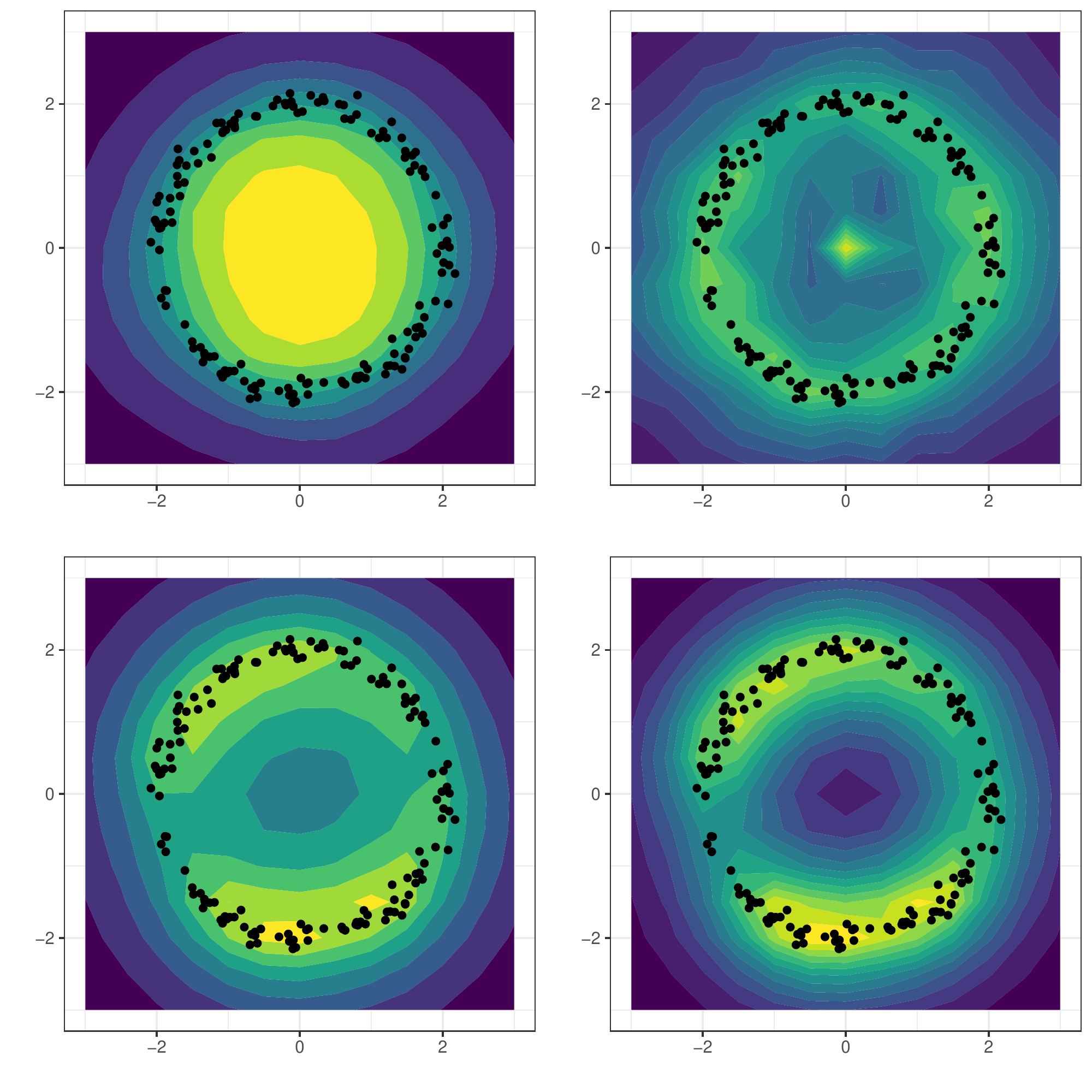}
    \caption{The depth contours of a circular data set of size $n = 150$ (the black points) under four different choices of distance for the metric spatial depth. The distances correspond, from left to right and from top to bottom: Euclidean distance, ISOMAP, RKHS with the rational quadratic kernel, RKHS with the Gaussian kernel. Yellow shades indicate the highest depths and the color scale is not uniform in the four panels}
    \label{fig:nonlinear}
\end{figure}

As a first option, we combine metric spatial depth with the idea behind ISOMAP \citep{tenenbaum2000global}, a classical method of manifold learning. ISOMAP works by using only local Euclidean distance information and disregarding the longer distances, which usually fail to capture any curved structure in the data. Namely, for each point $i$ in turn, we ``erase'' all but the $k$ smallest distances $d(x_i, x_j)$, $j \in \{ 1, \ldots , i - 1, i + 1 , \ldots , n \}$, where $k$ is a user-chosen parameter, and use the remaining distances to construct a weighted undirected graph $\mathcal{G}$ on the $n$ points. Next, we let $ d^*(x_i, x_j)$ denote the graph distance between $x_i, x_j$ in $\mathcal{G}$ (the shortest path between $x_i, x_j$ in the graph). The corresponding distances $d^*(x_i, \mu)$ between the sample and a ``test point'' $\mu$ are computed as graph distances after grafting $\mu$ to the graph $\mathcal{G}$ by its $k$ nearest neighbours. For small $k$ this results into $d^*$ that approximates the geodesic distances (assuming the points lie roughly on some manifold, as in our example). However, $k$ must also be kept large enough such that the graph $\mathcal{G}$ stays connected. The depth contours obtained with this procedure using $k = 8$ are shown in the upper right panel of Figure \ref{fig:nonlinear}. We observe that the contours capture the circular shape of the data, but assign large depths also close to the origin. This happens because the origin is roughly equally far away from all sample points, and its $k$ nearest neighbours are likely to be roughly uniformly distributed around the circle. Consequently $d^*(x_i, 0)$ cannot be too large for any $i$, giving the origin a large depth.


As a second alternative, we resort to the well-known ``kernel trick'' and Moore-Aronszajn theorem \citep{aronszajn1950theory} which says that, for every symmetric, positive definite kernel $\kappa: \mathcal{X} \times \mathcal{X} \rightarrow \mathbb{R}$, there exists a reproducing kernel Hilbert space (RKHS) $(\mathcal{H}, \langle \cdot, \cdot \rangle )$ such that $\kappa(x_1, x_2) = \langle \phi(x_1), \phi(x_2) \rangle$, for some ``feature mapping'' $\phi:\mathcal{X} \rightarrow \mathcal{H}$ determined implicitly by the choice of $\kappa$. Now, every Hilbert space is a metric space, meaning that the corresponding metric can be expressed through the kernel as $d^2(x_1, x_2) = \kappa(x_1, x_1) - 2 \kappa(x_1, x_2) + \kappa(x_2, x_2)$. Consequently, applying our proposed concept to these distances allows us to compute the metric spatial depths of the features $\phi(x_i)$. As the implicit mapping $\phi$ is typically non-linear, this allows capturing non-convexities in the original space $\mathcal{X}$.

As an example, we have used the rational quadratic kernel $\kappa(x_1, x_2) = (1 + \|x_1 - x_2\|^2)^{-1}$ (the bottom left panel of Figure \ref{fig:nonlinear}) and the Gaussian kernel $\kappa(x_1, x_2) = \exp(-0.933 \|x_1 - x_2\|^2)$ (the bottom right panel), where $\| x_1 - x_2 \|$ denotes the Euclidean distance. Inspecting the results shows that both kernels manage to capture the circular shape of the data and produce satisfactory contours. While the rational quadratic kernel still puts some depth in the origin, the Gaussian kernel, whose tuning parameter value 0.933 has been chosen manually, obtains contours that assign minimal depth to the origin. If depths are used as part of a supervised learning technique, standard cross-validation techniques could also be used to select any tuning parameter values involved in the distance function, see the next section for an example.

\subsection{Classification}

As our third example, we apply the metric spatial depth to a classification problem through the technique of depth-depth (DD) classification \citep{li2012dd}. Let $x_1, \ldots, x_{n_1} \in \mathcal{X}$ and $x_{n_1 + 1}, \ldots, x_{n_1 + n_2} \in \mathcal{X}$ denote samples from two groups and denote the corresponding empirical distributions by $P_{n1}$ and $P_{n2}$. In DD-classification, we compute the depth vectors $z_i := (D(x_i; P_{n1}), D(x_i; P_{n2}))$, $i \in  \{ 1, \ldots, n_1 + n_2 \}$ and fit a classification rule (such as a linear discriminant) to the two-dimensional sample $z_1, \ldots , z_{n_1 + n_2}$. A test point $x \in \mathcal{X}$ is then classified based on the vector $z := (D(x; P_{n1}), D(x; P_{n2}))$. This procedure extends to multiple groups in an obvious way. The main idea behind the DD-classifier is that depth functions automatically take into account the geometry of the data, avoiding the need to estimate any distributional parameters and making the method fully non-parametric


\begin{figure}[t!]
    \centering
    \includegraphics[width=1\textwidth]{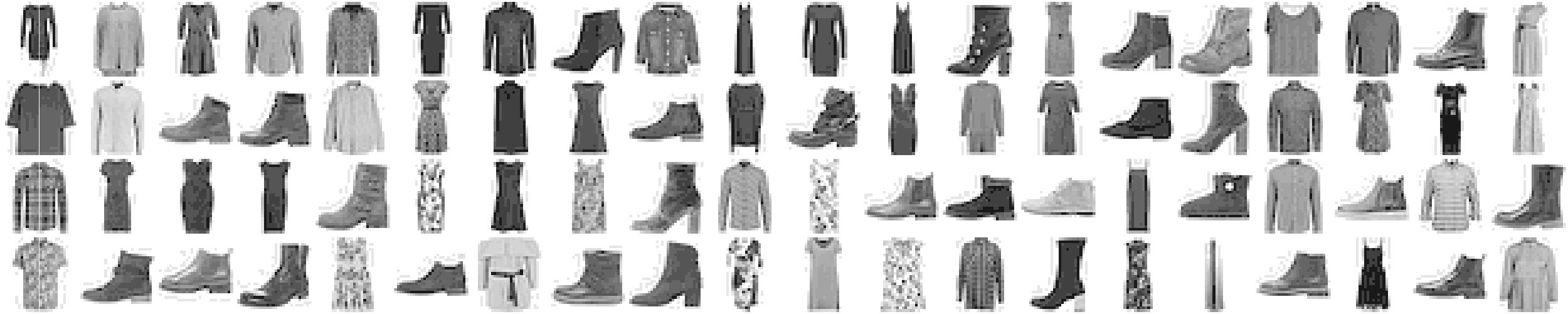}
    \caption{The first 80 images in the used data. The data consists of three classes: dresses, shirts and ankle boots.}
    \label{fig:lp_2}
\end{figure}

In this illustration, we use the Fashion MNIST dataset \citep{xiao2017fashion} available at Kaggle\footnote{https://www.kaggle.com/datasets/zalando-research/fashionmnist}. The data consists of $28 \times 28$ grayscale images of various clothing items grouped into 10 classes. For simplicity, we consider only the test part of the full data set and only three classes (dress, shirt, ankle boot). This gives us a working data set of $n = 3000$ observations of dimensionality $784$. The first 80 images in this set are shown on the left panel of Figure \ref{fig:lp_2}. We conduct a total of 100 replicates of our study and in each replicate draw randomly $150$ training images and $50$ test images from among the $3000$ observations. We use DD-classification with the metric spatial depth to classify the test images and use as our criterion the proportion of correct classifications in the test set. As a classifier we use either LDA or QDA. The main purpose of this illustration is to further demonstrate that while metric spatial depth is mainly targeted toward non-standard forms of data, it can still be used to enhance also the analysis of classical Euclidean data. We achieve this by choosing as our metric the $L_p$-distance with different values of $p = 0.5, 0.6, \ldots, 5$. The choice $p = 2$ thus corresponds to using the standard spatial depth \citep{chaudhuri1996geometric, vardi2000multivariate} in the analysis and the remaining choices correspond to novel methodology. See also \cite{perlibakas2004distance, rodrigues2018combining} for earlier uses of $L_p$-distances in the context of classification.

\begin{figure}[t!]
    \centering
    \includegraphics[width=1\textwidth]{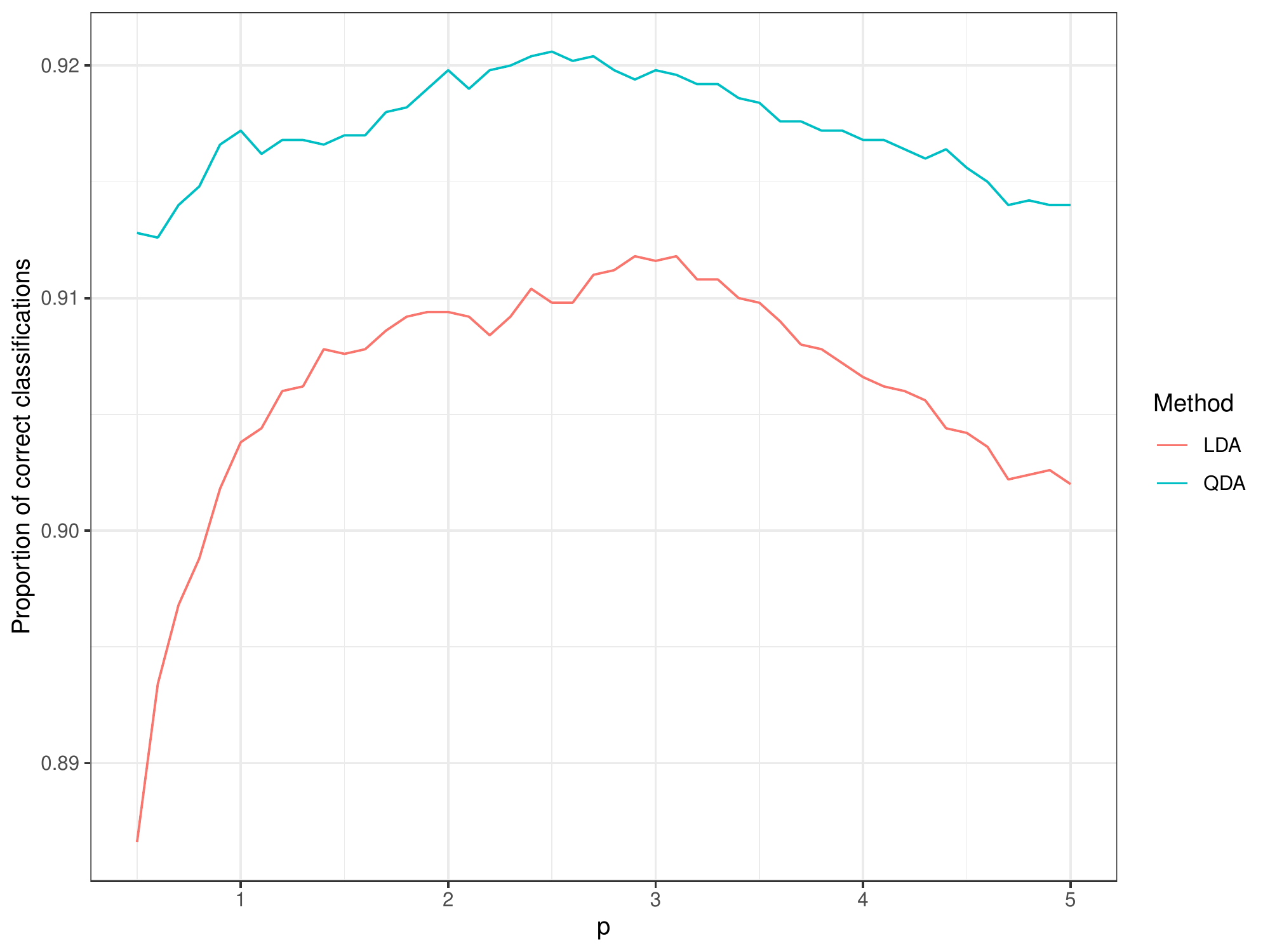}
    \caption{Average proportions of correct classifications as a function of the parameter $p$ of $L_p$-distance. The two curves correspond to the different classification methods.}
    \label{fig:lp_1}
\end{figure}


The lines in Figure \ref{fig:lp_1} depict the average proportions of correct classifications over the replicates as a function of $p$. We make two observations of interest: (i) The mean curves appear to be rather smooth functions of $p$ that achieve their maximum somewhere around $p \approx 3$ for LDA and $p \approx 2.5$ for QDA. In particular, for both LDA and QDA the maximum is achieved with a super-Euclidean geometry ($L_p$ with $p > 2$). As the extreme case $p = \infty$ corresponds to using only a single pixel's worth of information, this probably means that the classification information in the original images is to an extent concentrated to a small group of pixels. (ii) While the difference between LDA and QDA is minor, the latter appears the superior choice regardless of $p$, implying that the scatter plot of the points $z_i \in \mathbb{R}^3$ is not fully linearly separable. This is confirmed by Figure \ref{fig:lp_3} where we plot the paired scatter plot of the 3-dimensional training points $z_i$ for one replicate of the study with the choice $p = 3$. It is evident that while the separation between the groups is rather good, the point clouds still exhibit some curvature that make LDA suboptimal. Note also that the histograms on the diagonal of Figure \ref{fig:lp_3} reveal that the largest values (that is, depths) in each component of $z_i$ are indeed obtained by images of the corresponding class, as expected. To conclude, this (and the previous) example go to show that procedures targeted towards general object data can still benefit also the analysis of regular Euclidean data.

\begin{figure}[t!]
    \centering
    \includegraphics[width=1\textwidth]{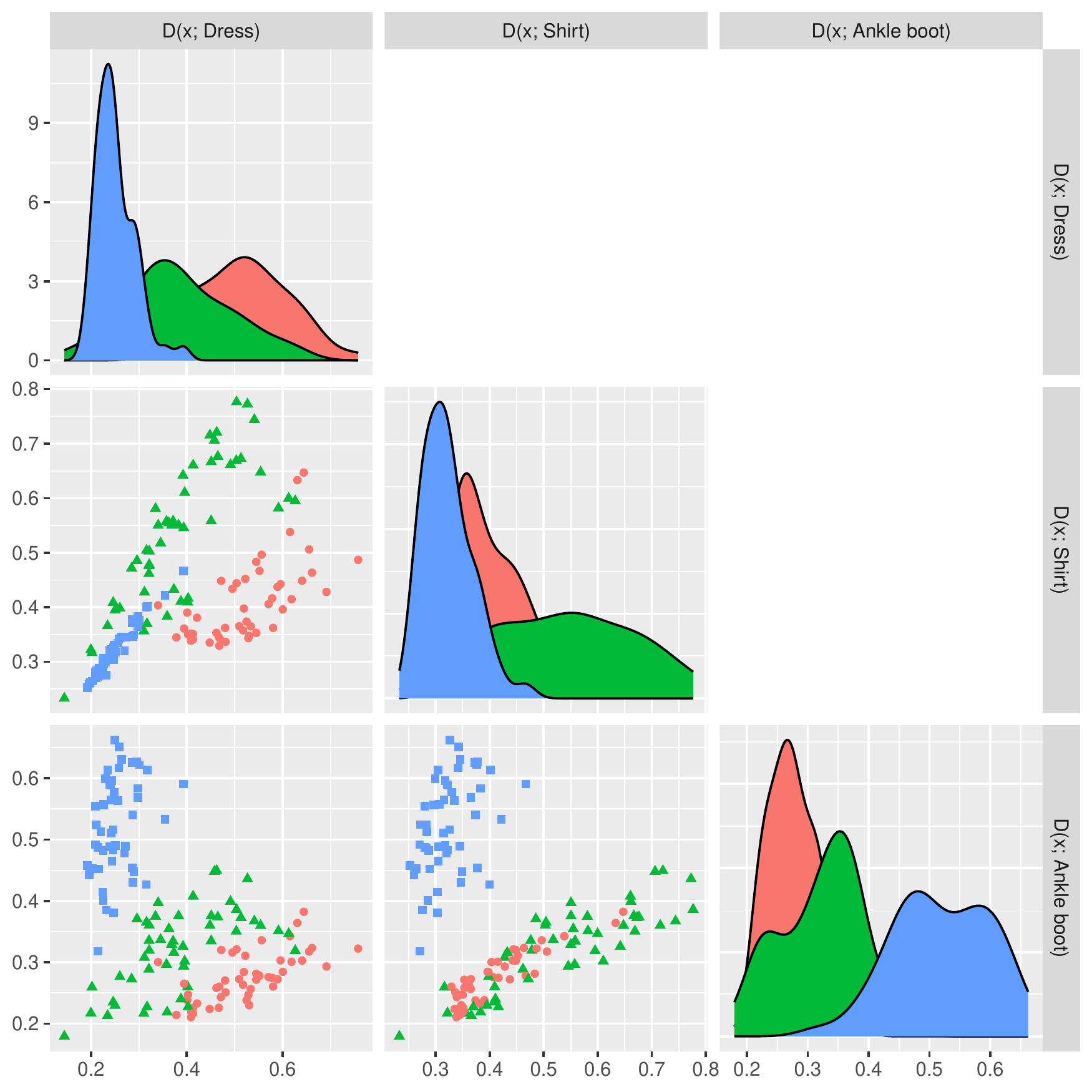}
    \caption{Paired scatter plots and histograms of the components of $z_i$ for one replicate of the study with the $L_3$-norm. The colors and shapes correspond to the three classes: dress (red dot), shirt (green triangle), ankle boot (blue square).}
    \label{fig:lp_3}
\end{figure}

\section{Discussion}

We close by commenting on possible avenues for further research. As with any measure of depth, a particularly interesting question regarding the metric spatial depth is finding the maximizers of the map $\mu \mapsto D(\mu; P_X)$ for a given probability distribution $P_X$. These maximizing points are typically called ``medians'' and, as location estimation plays an important role in almost all statistical practice, finding a way to carry out the maximization of $\mu \mapsto D(\mu; P_X)$ would be practically very valuable. For example, the clustering method $k$-means is based solely on finding centers (location estimates) of would-be clusters and any concept of ``robust metric location'' can be used to formulate a robust metric version of $k$-means.

When $(\mathcal{X}, d)$ is an Euclidean space, the points with maximal metric spatial depth are precisely those $\mu \in \mathcal{X}$ which satisfy  $\mathrm{E} \left\{ \mathrm{sgn}(X - \mu) \right\} = 0$, see Section \ref{subsec:example_hilbert}. The left-hand side of this condition is, under suitable regularity conditions for $X$, the gradient of the function $\mu \mapsto \mathrm{E} \| X - \mu \|$, showing that in the Euclidean case the spatial median is actually the $L_1$-equivalent of the mean vector. Based on this connection, it is natural to ask whether the same holds also in a general metric space. That is, are the minimizers of $\mu \mapsto \mathrm{E}\{ d(X, \mu) \}$ such that they also maximize $D(\mu; P_X)$? 

Finding the deepest points in practice, i.e., maximizing $\mu \mapsto D(\mu; P_n)$ is in the Euclidean case typically handled through gradient-based optimization, see \cite{vardi2000multivariate}. However, such techniques do not work in a general metric space due to the lack of vector space structure. In some special cases, alternative structural properties of $\mathcal{X}$ can possibly be used instead (e.g. geodesic convexity in case $\mathcal{X}$ admits a representation as a manifold). Finally, we remark that if the maximizer is searched only among the sample $x_1, \ldots, x_n$ (instead of over the full space $\mathcal{X}$), then, as described in Section \ref{sec:sample}, determining the deepest of these points is a simple enumeration task of order $\mathcal{O}(n^3)$. Given a sample of sufficient size, the deepest sample point can likely be used as a reasonable approximation to the actual deepest point.


\subsection*{Acknowledgements}

The work of JV was supported by the Academy of Finland (grants 347501 and 353769). The author has no competing interests to acknowledge.


\appendix



\section{Proofs}

\begin{proof}[Proof of Theorem \ref{theo:existence}]
Simplifying the squared reverse triangle inequality,
\begin{align}\label{eq:rev_triangle_ineq}
    d^2(X_1, X_2) \geq | d(X_1, \mu) - d(X_2, \mu) |^2,
\end{align}
gives
\begin{align*}
   d^2(X_1, \mu) + d^2(X_2, \mu) - d^2(X_1, X_2) \leq  2 d(X_1, \mu) d(X_2, \mu).
\end{align*}
Whereas, the squared triangle inequality,
\begin{align}\label{eq:triangle_ineq}
    d^2(X_1, X_2) \leq \{ d(X_1, \mu) + d(X_2, \mu) \}^2,
\end{align}
gives,
\begin{align*}
   d^2(X_1, \mu) + d^2(X_2, \mu) - d^2(X_1, X_2) \geq  -2 d(X_1, \mu) d(X_2, \mu).
\end{align*}
Hence,
\begin{align*}
    | D(\mu; P_X) | \leq 1 + \frac{1}{2} \mathrm{E} \left\{ \frac{| d^2(X_1, \mu) + d^2(X_2, \mu) - d^2(X_1, X_2) | }{d(X_1, \mu) d(X_2, \mu) } \right\} \leq 2,
\end{align*}
concluding the proof.
\end{proof}

\begin{proof}[Proof of Theorem \ref{theo:influence_function}]
    The claimed form for the influence function follows after observing that,
    \begin{align*}
        & D(\mu; (1 - \varepsilon) P_X + \varepsilon \delta_z) \\
        =& 1 - \frac{1}{2} \left[ 2 (1 - \varepsilon)^2 \{ 1 - D(\mu; P_X) \} + 2 \varepsilon (1 - \varepsilon) \mathrm{E} \{ h(X, z, \mu) \} + 2 \varepsilon^2 \mathbb{I}(\mu \neq z) \right],
    \end{align*}
    and that $| \mathrm{E} \{ h(X, z, \mu) \} | \leq 2$ by the proof of Theorem \ref{theo:existence}. Furthermore, the same bound $| \mathrm{E}\{ h(X, z, \mu) \} | \leq 2$, when combined with Theorem \ref{theo:range}, yields the second part of the current theorem.
\end{proof}

\begin{proof}[Proof of Theorem \ref{theo:range}]
    That $D(\mu; P_X)$ takes values in $[0, 2]$ follows directly from the proof of Theorem \ref{theo:existence}. For part (i), we observe that $D(\mu ; P_X) = 0$ precisely when
    \begin{align*}
        \mathbb{I}(X_1 \neq \mu \cap X_2 \neq \mu) \frac{ d^2(X_1, \mu) + d^2(X_2, \mu) - d^2(X_1, X_2) }{d(X_1, \mu) d(X_2, \mu) } = 2,
    \end{align*}
    almost surely. By the definition of the indicator function and the proof of Theorem \ref{theo:existence}, this occurs if and only if $P_X(\{\mu\}) = 0$ and if equality is achieved in the squared reverse triangle inequality \eqref{eq:rev_triangle_ineq} almost surely. The latter of these is equivalent to the statement that $P_X ( L[X_1, X_2, \mu] \cup L[X_2, X_1, \mu] ) = 1$, proving claim (i).

    Part (ii) is proven similarly, but by invoking the squared triangle inequality \eqref{eq:triangle_ineq} instead of \eqref{eq:rev_triangle_ineq}.
\end{proof}

\begin{proof}[Proof of Theorem \ref{theo:divergent_sequence}]
    We begin by establishing some terminology. A sequence $A_n$ of random variables in $\mathbb{R}$ is said to be a \textit{D-sequence} (D is for ``diverging'') if $P(|A_n| \leq M) \rightarrow 0$ for every $M > 0$. Whereas, a random variable $B$ taking values in $\mathbb{R}$ is said to be bounded if, for all $\varepsilon > 0$, there exists $K > 0$ such that $P(|B| \geq K) < \varepsilon$.

    We next show that for any bounded random variable $B$ and any D-sequence $A_n$, the quotient $B/A_n$ converges to zero in probability. To see this, fix $\Delta, \varepsilon > 0$ and pick $K > 0$ such that $P(|B| \geq K) < \varepsilon$, and $n_0$ such that, for $n > n_0$, we have $P(|A_n| \leq K/\Delta) < \varepsilon$. Then,
    \begin{align*}
        & P(|A_n| \leq |B|/\Delta) \\
        =& P(|A_n| \leq |B|/\Delta \mid |B| \leq K ) P(|B| \leq K) \\
        +& P(|A_n| \leq |B|/\Delta \mid |B| > K ) P(|B| > K). 
    \end{align*}
    Now, $P(|A_n| \leq |B|/\Delta \mid |B| \leq K ) \leq P(|A_n| \leq K/\Delta \mid |B| \leq K )$, allowing us to write,
    \begin{align*}
        & P(|A_n| \leq |B|/\Delta) \\
        \leq& P(|A_n| \leq K/\Delta)\\
        +& \{ P(|A_n| \leq |B|/\Delta \mid |B| > K ) - P(|A_n| \leq K/\Delta \mid |B| > K ) \} P(|B| > K) \\
        \leq& P(|A_n| \leq K/\Delta) + 2 P(|B| > K) \\
        \leq& 3 \varepsilon,
    \end{align*}
    showing that $B/A_n = o_p(1)$.
    
    Fix now an arbitrary point $\mu_0 \in \mathcal{X}$. It is clear that $d(X_1, \mu_0)$ is bounded. We next show that $d(X_1, \mu_n)$ is a D-sequence. Let $\varepsilon > 0$, $M > 0$ be arbitrary and let $K > 0$ be such that $P\{ d(X_1, \mu_0) \geq K \} < \varepsilon$. Pick $n_0$ such that, for all $n > n_0$, we have $d(\mu_0, \mu_n) \geq M + K$.

    Then, for $n > n_0$, we have, by the reverse triangle inequality, that
    \begin{align*}
        P\{ d(X_1, \mu_n) \leq M \} \leq& P \{ d(X_1, \mu_0) \geq d(\mu_0, \mu_n) - M \} \\
        \leq& P \{ d(X_1, \mu_0) \geq K \} \\
        \leq& \varepsilon,
    \end{align*}
    showing that $d(X_1, \mu_n)$ is a D-sequence. Finally, it is obvious that $d(X_1, X_2)$ is bounded.

    Consider now the quantity $R_{n4} := -d^2(X_1, X_2)/\{ d(X_1, \mu_n) d(X_2, \mu_n) \}$. A straightforward computation reveals that the product of $D$-sequences is a $D$-sequence and, hence, we have that $R_{n4} = o_p(1)$. Next, by the triangle inequality, the quantity $R_{n2} := d(X_1, \mu_n)/d(X_2, \mu_n)$ satisfies
    \begin{align*}
        \left\{ 1 + \frac{d(X_1, X_2)}{d(X_1, \mu_n)} \right\}^{-1} \leq R_{n2} \leq 1 + \frac{d(X_1, X_2)}{d(X_2, \mu_n)},
    \end{align*}
    showing that $R_{n2} = 1 + o_p(1)$. We can similarly show that $R_{n3} := d(X_1, \mu_n)/d(X_2, \mu_n) = 1 + o_p(1)$. Finally, it is straightforwardly proven that $R_{n1} := \mathbb{I}(X_1 \neq \mu_n \cap X_2 \neq \mu_n) = 1 + o_p(1)$. Consequently, denoting $R_n := R_{n1}(R_{n2} + R_{n3} + R_{n4}) = 2 + o_p(1)$, we have,
    \begin{align*}
        D(\mu; P_X) = 1 - \frac{1}{2} \mathrm{E}(R_n).
    \end{align*}
    As $|R_n|$ is uniformly bounded in $n$, we obtain the desired claim that $D(\mu; P_X) \rightarrow 0$.
\end{proof}

\begin{proof}[Proof of Theorem \ref{theo:continuity_1}]
    Let $X_{1n}, X_{2n} \sim P_n$ be independent and let $X_1, X_2 \sim P_X$ be independent. Then $(X_{1n}, X_{2n}) \rightsquigarrow (X_1, X_2)$. Now, since the map $g: \mathcal{X}^2 \rightarrow \mathbb{R}$ defined as
    \begin{align*}
        g(x_1, x_2) := \mathbb{I}(x_1 \neq \mu \cap x_2 \neq \mu) \frac{ d^2(x_1, \mu) + d^2(x_2, \mu) - d^2(x_1, x_2) }{d(x_1, \mu) d(x_2, \mu) },
    \end{align*}
    is $P_X$-a.s. continuous, the continuous mapping theorem guarantees that
    \begin{align*}
        g(X_{1n}, X_{2n}) \rightsquigarrow g(X_{1}, X_{2}),
    \end{align*}
    as $n \rightarrow \infty$ (note that $g$ is measurable since it is $P_X$-a.s. equal to a continuous function). The function $g$ is bounded by the proof of Theorem \ref{theo:existence}, implying that
    \begin{align*}
        \mathrm{E} \{ g(X_{1n}, X_{2n}) \} \rightarrow \mathrm{E} \{ g(X_{1}, X_{2}) \},
    \end{align*}
    and concluding the proof.
\end{proof}

\begin{proof}[Proof of Lemma \ref{lem:inner_product}]
    Plugging in the inner products, $d^2(a, b) = \langle a - b, a - b \rangle$, to the definition of $D(\mu; P_X)$, we obtain
    \begin{align}\label{eq:hilbert_space_form}
        D(\mu; P_X) = 1 - \mathrm{E} \left\{ \left\langle \mathbb{I}(X_1 \neq \mu) \frac{X_1 - \mu}{\| X_1 - \mu \|}, \mathbb{I}(X_2 \neq \mu) \frac{X_2 - \mu}{\| X_2 - \mu \|} \right\rangle \right\}.
    \end{align}
    Now, for two independent random elements $Y_1, Y_2$ in $\mathcal{X}$ whose expected values exist, we have,
    \begin{align*}
        \mathrm{E} ( \langle Y_1, Y_2 \rangle ) &= \mathrm{E} \{ \mathrm{E} ( \langle Y_1, Y_2 \rangle \mid Y_2 ) \} \\
        &= \mathrm{E} \{ \langle \mathrm{E}(Y_1), Y_2 \rangle \} \\
        &= \langle \mathrm{E}(Y_1), \mathrm{E}(Y_2) \rangle.
    \end{align*}
    Consequently, \eqref{eq:hilbert_space_form} takes the desired form
    \begin{align*}
        D(\mu; P_X) = 1 - \left\| \mathrm{E} \left\{ \mathrm{sgn}(X - \mu) \right\} \right\|^2.
    \end{align*}
\end{proof}

\begin{proof}[Proof of Lemma \ref{lem:rail_optimality}]
    The depth $D(x_1; P_X)$ equals,
    \begin{align*}
        1 - \frac{1}{2n^2} \sum_{i = 2}^n \sum_{j = 2}^n \frac{ d^2(x_i, x_1) + d^2(x_j, x_1) - d^2(x_i, x_j) }{d(x_i, x_1) d(x_j, x_1) }.
    \end{align*}
    Now, each index pair with $i = j$ is easily checked to contribute $2$ to the sum, whereas each pair of distinct indices contributes the summand,
    \begin{align}\label{eq:rail_individual_summand}
        \frac{ d^2(x_i, x_1) + d^2(x_j, x_1) - d^2(x_i, x_j) }{ d(x_i, x_1) d(x_j, x_1) }
    \end{align}
    By the proof of Theorem \ref{theo:range}, the quantity \eqref{eq:rail_individual_summand} is lower bounded by $-2$ and this bound is reached precisely if $d(x_i, x_j) = d(x_i, x_1) + d(x_1, x_j)$. The claim now follows.
\end{proof}

\begin{proof}[Proof of Lemma \ref{lem:finite_set}]
    The closed form for $D(1; P_X)$ is straightforwardly derived. For the lower bound, we observe that $s_2 := \sum_{i = 1}^n p_i^2 \leq 1$ with equality if and only if all but one of the $p_i$ equal zero. Consequently, $(1/2) ( 1 - s_2 ) + p_1 \geq 0$ with equality if and only if exactly one of $p_2, \ldots, p_n$ equals 1. For the upper bound, we first write $D(1; P_X) = 1 - (1/2) (1 - p_1)^2 - (1/2) (s_2 - p_1^2)$ from which it follows that $D(1; P_X) \leq 1$ with equality if and only if $p_1 = 1$.
\end{proof}

\begin{proof}[Proof of Theorem \ref{theo:cycle_graph}]
    We first compute the depth of an arbitrary point for a discrete, equispaced uniform distribution on $n = 2k$ points on the unit circle. We denote this distribution by $P_n$.
    
    We label the $2k$ points such that our point of interest is $0$ and the cyclical ordering of the points is $-(k - 1), -(k - 2), \ldots , -1 , 0 , 1 , \ldots , (k - 1), k$. That is, the point opposite of $0$ on the cycle is $k$. Now, the expectation in the definition of $D(0; P_n)$ equals,
    \begin{align}\label{eq:cycle_form_1}
        \frac{1}{4k^2} \sum_{i} \sum_j \frac{ d^2(i, 0) + d^2(j, 0) - d^2(i, j) }{d(i, 0) d(j, 0) },
    \end{align}
    where the summations range over all points on the cycle but $0$. To evaluate this sum, we first fix $i = 1, \ldots, k - 1$ and observe that $j$ can be chosen to lay in four different regions of the cycle, depending on whether the distances $d(j, 0)$ and $d(i, j)$ are realized clockwise or counterclockwise. Going individually through all four cases we observe that, for a fixed value of $i = 1, \ldots, k - 1$, summing over $j$ contributes the quantity
    \begin{align*}
       \frac{1}{4k^2} \left\{ 2i  + \sum_{j = 0}^i \frac{4ki + 4k(k - j) - 2i(k - j) - 4k^2}{i(k - j)} \right\}
    \end{align*}
    to the expectation. By symmetry, each $i = -1, \ldots, -(k - 1)$ contributes the same quantity as $|i|$ and, for $i = k$, the contribution is, using similar technique, seen to be $(2k - 1)/(2k^2)$. Consequently, \eqref{eq:cycle_form_1} takes (after simplification) the form,
    \begin{align*}
        4 - \frac{2}{k} - \frac{3}{2 k^2} + \frac{2}{k} H_{k - 1} - 2 \sum_{i = 1}^{k - 1} \sum_{j = 0}^i \frac{1}{i(k - j)},
    \end{align*}
    where $H_\ell$ denotes the $\ell$th harmonic number. After observing that the double sum above equals $H_{k - 1}/k + \sum_{j = 1}^{k - 1} (H_{k - 1} - H_{k - 1 - j})/j$, the expression simplifies to
    \begin{align}\label{eq:cycle_form_2}
        4 - \frac{2}{k} - \frac{3}{2 k^2} - 2 (H_{k - 1})^2 + 2 \sum_{j = 1}^{k - 2} \frac{1}{k - 1 - j} H_j.
    \end{align}
    As an intermediate result, we next show that, for all $n \geq 1$,
    \begin{align}\label{eq:cycle_intermediate_1}
        \sum_{j = 1}^{n} \frac{1}{n + 1 - j} H_j = (H_{n + 1})^2 - \sum_{j = 1}^{n + 1} \frac{1}{j^2}.
    \end{align}
    To see this, we write
    \begin{align*}
        \sum_{j = 1}^{n} \frac{1}{n + 1 - j} H_j &= \sum_{j = 1}^{n} \sum_{\ell = 1}^j \frac{1}{\ell (j + 1 - \ell)} \\
        &= \sum_{j = 1}^{n} \sum_{\ell = 1}^j \frac{1}{j + 1} \left( \frac{1}{\ell} -  \frac{1}{j + 1 - \ell} \right) \\
        &= 2 \sum_{j = 1}^n \frac{1}{j + 1} H_j \\
        &= 2 \sum_{j = 1}^{n + 1} \frac{1}{j} H_j - 2 \sum_{j = 1}^{n + 1} \frac{1}{j^2}.
    \end{align*}
    The relation \eqref{eq:cycle_intermediate_1} now follows after observing that
    \begin{align*}
        2 \sum_{j = 1}^{n + 1} \frac{H_j}{j} = \sum_{j = 1}^{n + 1} \frac{1}{j^2} + (H_{n + 1})^2,
    \end{align*}
    by the example on page 850 in \cite{spiess1990some}. Plugging in \eqref{eq:cycle_intermediate_1} to \eqref{eq:cycle_form_2} now yields
    \begin{align*}
         D(0 ; P_n) = -1 + \frac{1}{k} - \frac{1}{4 k^2} + \sum_{j = 1}^{k} \frac{1}{j^2}.
    \end{align*}
    By the classical Basel problem, as $n \rightarrow \infty$, the depth $D(0; P_n)$ approaches the value $\pi^2/6 - 1$. Hence, the proof is concluded as soon as we show that $P_n \rightsquigarrow P_X$ and invoke Theorem \ref{theo:continuity_1}.

    To see this, note that the weak convergence of distributions on $\mathbb{S}^1$ is equivalent to the pointwise convergence of the corresponding characteristic functions. That the characteristic function of $P_n$ converges pointwise to that of $P_X$ now follows by equations 3.5.14 and 3.5.16 in \cite{mardia2000directional}.

\end{proof}

\vskip 0.2in
\bibliographystyle{apalike}
\bibliography{references}

\end{document}